\documentclass[a4paper,11pt]{amsart}

\usepackage{amsmath, amsthm, amsfonts, amssymb}
\usepackage[bookmarks=false]{hyperref}
\usepackage{enumerate}

\usepackage{color}

\numberwithin{equation}{section}

\newtheorem{letterthm}{Theorem}

\newtheorem{lettercor}[letterthm]{Corollary}

\newtheorem{thm}{Theorem}[section]
\newtheorem{lem}[thm]{Lemma}

\newtheorem{cor}[thm]{Corollary}
\newtheorem{prop}[thm]{Proposition}

\theoremstyle{definition}
\newtheorem{rem}[thm]{Remark}

\newtheorem{df}[thm]{Definition}

\newtheorem*{question}{Question}

\newcommand{\R}{\mathbf{R}}
\newcommand{\C}{\mathbf{C}}
\newcommand{\Z}{\mathbf{Z}}
\newcommand{\F}{\mathbf{F}}

\newcommand{\N}{\mathbf{N}}
\newcommand{\B}{\mathbf{B}}

\newcommand{\IC}{\mathbb C}

\newcommand{\IN}{\mathbb N}

\newcommand{\cH}{\mathcal{H}}
\newcommand{\cK}{\mathcal{K}}

\newcommand{\cI}{\mathcal{I}}

\newcommand{\fM}{\mathfrak{M}}

\newcommand{\Ad}{\operatorname{Ad}}
\newcommand{\id}{\text{\rm id}}

\newcommand{\Inn}{\operatorname{Inn}}
\newcommand{\Aut}{\operatorname{Aut}}

\newcommand{\rL}{\mathord{\text{\rm L}}}

\newcommand{\conv}{\overline{\mathord{\text{\rm conv}}} \,}

\newcommand{\rE}{\mathord{\text{\rm E}}}

\newcommand{\alg}{\text{alg}}

\newcommand{\ovt}{\mathbin{\overline{\otimes}}}

\newcommand{\ve}{\varepsilon}
\newcommand{\vp}{\varphi}

\newcommand{\ip}[1]{\mathopen{\langle}#1\mathclose{\rangle}}

\newcommand{\op}{^{\mathrm{op}}}

\newcommand{\II}{{\rm II}}
\newcommand{\III}{{\rm III}}

\begin{document}
\title[Full Factors and Co-Amenable Inclusions]{Full Factors and Co-amenable Inclusions}

\author{Jon Bannon}
\address{Siena College Department of Mathematics, 515 Loudon Road, Loudonville,
NY\ 12211, USA}
\email{jbannon@siena.edu}
\thanks{JB is partially supported by a Lancaster University (STEM) Fulbright Scholar Award}
\author{Amine Marrakchi}
\thanks{AM is a JSPS International Research Fellow (PE18760)}
\address{RIMS, Kyoto University, 606-8502 Japan}
\email{amine.marrakchi@math.u-psud.fr}
\author{Narutaka Ozawa}
\address{RIMS, Kyoto University, 606-8502 Japan}
\email{narutaka@kurims.kyoto-u.ac.jp}
\thanks{NO is partially supported by JSPS KAKENHI 17K05277 and 15H05739}
\keywords{von Neumann Algebras; Full factor; co-amenable inclusion; compact group; minimal action}
\subjclass[2010]{46L10, 46L36, 46L40, 46L55}

\begin{abstract}
We show that if $M$ is a full factor and $N \subset M$ is a co-amenable subfactor with expectation, then $N$ is also full. This answers a question of Popa from 1986. We also generalize a theorem of Tomatsu by showing that if $M$ is a full factor and $\sigma \colon G \curvearrowright M$ is an outer action of a compact group $G$, then $\sigma$ is automatically minimal and $M^G$ is a full factor which has w-spectral gap in $M$.
Finally, in the appendix, we give a proof of the fact that several natural notions of co-amenability for an inclusion $N\subset M$ of von Neumann algebras are equivalent, thus closing the cycle of implications given 
in Anantharaman-Delaroche's paper in 1995.
\end{abstract}

\maketitle

\section{Introduction}

An inclusion of two von Neumann algebras $N \subset M$ is called \emph{co-amenable} if the inclusion of the commutants $M' \subset N'$ admits a conditional expectation $\Phi \colon N' \rightarrow M'$. In particular, $M$ is amenable if and only if the inclusion $\C \subset M$ is co-amenable. This notion originally appeared in Definition 3.2.1 of \cite{Po86}, and was there called relative amenability. In \cite{MoPo03}, the relation between this concept and the notion of co-amenability for groups was clarified. Indeed, an inclusion of groups $H \subset G$ co-amenable if and only if there exists a $G$-invariant mean on $\ell^\infty(G/H)$, and it holds that the inclusion of group von Neumann algebras $L(H) \subset L(G)$ is co-amenable if and only if the inclusion $H \subset G$ is co-amenable. Other examples of co-amenable inclusion are given by finite index inclusions, inclusions of the form $N\subset N\rtimes G$ with $G$ amenable or inclusions of the form $M^G \subset M$ for a minimal compact group action $G \curvearrowright M$.

Following \cite{Co74}, we say that a separable factor $M$ is \emph{full} if the group of all its inner automorphisms $\Inn(M)$ is closed in $\Aut(M)$. In the type $\II_1$ case, the factor $M$ is full if and only if it does not have the \emph{property Gamma} of Murray and von Neumann \cite{MvN43}, i.e.\ if and only if every central sequence of $M$ is trivial, or equivalently, if and only if $M' \cap M^\omega = \C$ where $M^\omega$ is the ultrapower of $M$ with respect to a free ultrafilter on $\N$. Fullness is also closely related to a group-theoretic notion. Recall that a group $G$ is \emph{inner amenable} when there exists a non-trivial conjugacy invariant mean on $\ell^\infty(G)$. By \cite{Ef73}, if $G$ is non-inner amenable then $L(G)$ is a full factor. A subtle counter-example due to Vaes \cite{Va09} shows that the converse implication is however not true.

It is not difficult to see that if $H \subset G$ is a co-amenable group inclusion and if $H$ is inner amenable, then $G$ must also be inner amenable. In \cite[Problem 3.3.2]{Po86}, Popa asked the following analogous question in the context of von Neumann algebras: 

\begin{question}[\cite{{Po86}}] If $M$ is a full $\II_1$ factor and $N \subset M$ is a co-amenable subfactor, is it true that $N$ is also full? 
\end{question}

This question was partially motivated by Proposition 1.4.1 of \cite{Po83}, which affirmatively answers the above question when $M$ is the crossed product of $N$ by an action of $\Z$. A few years later (cf.\ Proposition 1.11 of \cite{PiPo86}), he and Pimsner proved that when $N\subset M$ is a finite index inclusion of $\II_1$ factors then $N$ is full if and only if $M$ is full. The case where the co-amenable inclusion $N\subset M$ is regular and irreducible was confirmed by B\'{e}dos \cite{Be90}, and independently by Bisch \cite{Bi90}. Such subfactors arise as cocycle crossed products of free cocycle actions of amenable groups on $\II_{1}$ factors (cf.\ \cite{Ch79}), and both B\'{e}dos's and Bisch's proofs rely heavily on this. We note that the specialized question for crossed products by amenable groups also follows from the recent work of the second author, which completely characterizes fullness of crossed product factors $N\rtimes G$ with $N$ a factor of arbitrary type and $G$ discrete amenable \cite{Ma18}. In that case, not only $N$ is full, but it also has \emph{spectral gap} in $N \rtimes G$.

We now state our main theorem. We assume that our inclusion $N \subset M$ is \emph{with expectation}, i.e.\ that there exists a faithful normal conditional expectation $\rE_N\colon M \rightarrow N$. This is automatic if $M$ is tracial, but it is both a necessary and a natural assumption in the non-tracial setting.

\begin{letterthm} \label{co-amenable full}
Let $M$ be a separable full factor and $N \subset M$ a co-amenable subalgebra with expectation. Then there exists a non-zero projection $p \in N' \cap M$ such that $p(N' \cap M^\omega)p= \C p$ for all $\omega \in \beta \N$.
\end{letterthm}

In particular, we obtain a complete solution to Popa's question, but for factors of arbitrary type and with the additional property that $N$ has \emph{w-spectral gap} in $M$ in the irreducible case.
\begin{lettercor} \label{cor w-gap}
Let $M$ be a separable full factor and let $N \subset M$ be a co-amenable subfactor with expectation. Then $N$ is full. Moreover, if $N'\cap M=\C$, then we have $N' \cap M^\omega=\C$ for all $\omega \in \beta \N$.
\end{lettercor}

 We note that our proof is completely different from the case of crossed products by amenable groups. Indeed, in that specific case, Bisch proves that if $N$ is not full then $M' \cap N^\omega \neq \C$. But he subsequently (cf.\ \cite{Bi94}) exhibits a finite index inclusion $N\subset M$ of hyperfinite $\II_1$ factors such that $M' \cap N^\omega =\C$, hence showing that one cannot hope to solve the general problem by proving that $M' \cap N^\omega \neq \C$ for all co-amenable inclusions of $\II_{1}$ factors $N\subset M$ where $N$ is not full. Instead, the proof of Theorem \ref{co-amenable full} essentially encodes a reduction of the problem to the finite index case. We also note that in the case of crossed products of $\II_1$ factors by amenable groups, one can actually show that the subfactor $N \subset M$ has spectral gap thanks to \cite{Jo81} and \cite{Ma18}. In Remark \ref{rem spectral gap}, we observe that this is no longer true for arbitrary co-amenable subfactors: in general the conclusion $N' \cap M^\omega=\C$ of Corollary \ref{cor w-gap} cannot be improved to true spectral gap.

In our second main result, we give an application of Theorem \ref{co-amenable full} to actions of compact groups on full factors. Very recently,  Tomatsu proved that if $\sigma \colon G \curvearrowright M$ is a minimal action of a compact group $G$ on a full factor $M$, then $M^G$ and $M \rtimes G$ are full factors \cite[Theorem 4.8]{To18}. Recall that an action is \emph{minimal} if it is faithful and $(M^G)' \cap M=\C$. A minimal action is necessarily \emph{outer} meaning that $\sigma_g \notin \Inn(M)$ for all $g \neq 1$. But the converse is far from being true in general. Nevertheless, thanks to Theorem \ref{co-amenable full}, we can show that outerness automatically implies minimality for actions of compact groups on full factors. This is the content of our second main theorem. We also strengthen Tomatsu's result by showing that $M^G$ has w-spectral gap in $M$.

\begin{letterthm} \label{minimalGaction}
Let $\sigma \colon G \curvearrowright M$ be an outer action of a compact group $G$ on a separable factor $M$. Assume that $M$ is full. Then $\sigma$ is automatically minimal and $M^G$ and $M \rtimes G$ are full factors. Moreover, we have $(M^G)' \cap M^\omega=\C$ and $M ' \cap (M \rtimes G)^\omega = \C$ for all $\omega \in \beta \N$.
\end{letterthm}

To obtain this theorem, we first prove the following dichotomy: for any outer action $\sigma \colon G \curvearrowright M$ of a compact group $G$ on an arbitrary factor $M$, either $(M^G)' \cap M =\C$ or $(M^G)' \cap  M$ is diffuse. Then we apply Theorem \ref{co-amenable full}.

\subsection*{Acknowledgments}
The first author thanks Jan Cameron for sharing his ideas on Property $\Gamma$ for an earlier approach to the problem, and is grateful for the hospitality of the Lancaster University Department of Mathematics and Statistics.  The second author is grateful to Yuki Arano for a thought-provoking discussion regarding Theorem \ref{minimal or diffuse}. We are also grateful to Adrian Ioana for providing us with the reference \cite{HK05} used in Remark \ref{rem spectral gap}.


\section{Preliminaries}

\subsection{Ultrapowers}
Given a free (i.e.\ nonprincipal) ultrafilter $\omega \in \beta \N \setminus \N$ and $\sigma$--finite von Neumann algebra $M$ one defines 
\begin{align*}
\cI_{\omega}(M)&=\{ (x_{n})_{n}\in \ell^{\infty}(\N,M) \mid \lim_{n\rightarrow \omega} x_{n}=0 \; *\text{-strongly} \}
\end{align*}
and its multiplier algebra  
\begin{align*}
\fM^{\omega}(M)&=\{(x_{n})_{n}\in \ell^{\infty}(\N,M) \mid (x_{n})_{n} \cI_{\omega}(M) + \cI_{\omega}(M)(x_{n})_{n} \subset \cI_{\omega}(M)\}.
\end{align*}
The latter is a $C^{*}$-algebra inside of which the former is a norm-closed two-sided ideal. The quotient $C^{*}$-algebra $M^{\omega}:=\fM^{\omega}(M)/\cI_{\omega}(M)$ is actually a von Neumann algebra and is called the Ocneanu ultrapower of $M$ associated to $\omega$ \cite{Oc85}. 
The von Neumann algebra $M$ naturally embeds into $M^{\omega}$ as $x\mapsto (x_{n})_{n}$ with $x_{n}=x$ for all $n\in \N$. 

The Ocneanu ultrapower generalizes the tracial ultrapower associated to a finite von Neumann algebra $M$ with trace $\tau$. Indeed, in that specific case we have $\cI_{\omega}(M)=\{(x_{n})_{n} \in \ell^{\infty}(\N,M) \mid \lim_{n\rightarrow \omega}\|x_{n}\|_{2}=0\}$ and $\fM^{\omega}(M)=\ell^{\infty}(\N,M)$, and the Ocneanu ultrapower admits the faithful normal trace $\tau_{\omega}((x_{n})_{n}+\cI_{\omega}(M))=\lim_{n\rightarrow \omega}\tau(x_{n})$.

\subsection{Standard form and basic construction}
For any von Neumann algebra $M$, we denote by $(M, \rL^2(M), J, \rL^{2}(M)^{+})$ its \emph{standard form} \cite{Ha73}. Recall that $\rL^2(M)$ is naturally endowed with the structure of a $M$-$M$-bimodule: we will simply write $x \xi y = x Jy^*J \xi$ for all $x, y \in M$ and all $\xi \in \rL^2(M)$. We will denote by $\lambda \colon M \rightarrow \B(\rL^2(M))$ and $\rho \colon M\op \rightarrow \B(\rL^2(M))$ the associated left and right regular representation of $M$. We have $\lambda(M)=\rho(M\op)'$. We will use the notation  $C_{\lambda \cdot \rho}^{*}(M)$ for the $C^*$-algebra generated by $\lambda(M)\rho(M\op)$.

If $N \subset M$ is a subalgebra, the so-called \emph{basic construction} associated to it is the von Neumann algebra $\rho(N\op)'$, which will be denoted $\langle M, N \rangle$. Since $\lambda(M) \subset \rho(N\op)'$, we will view $M$ as a subalgebra of $\langle M, N \rangle$. When $N \subset M$ admits a faithful normal conditional expectation $\rE_N \colon M \rightarrow N$, there is associated to it a canonical $N$-$N$-bimodular isometry $V \colon \rL^2(N) \rightarrow \rL^2(M)$. Its range projection $e_N=VV^* \in \B(\rL^2(M))$ is the \emph{Jones projection} of $\rE_N$. In that case, it holds that $e_Nxe_N=\rE_N(x)e_N$ for all $x \in M$ and $Me_NM$ is a dense $*$-sub\-algebra of $\langle M, N \rangle$. Moreover, we can identify $\rL^2(\langle M, N \rangle )$ with the space $\rL^{2}(\rE_N)$ associated to the completely positive map $\rE_N \colon M\rightarrow M$ via the Stinespring construction, i.e.\ as the Hilbert space separation and completion of $M\odot \rL^{2}(M)$ under the semi-inner product defined on simple tensors by $\langle x_{1}\otimes \eta_{1},x_{2}\otimes \eta_{2} \rangle_{\rE_N}:=\langle\rE_N(x_{2}^{*}x_{1})\eta_{1} ,\eta_{2}\rangle_{\rL^{2}(M)}$ with the natural bimodule action given by $a(x\otimes \eta)b:=ax\otimes\eta b$ for $x_{1},x_{2},a,b,x\in M$ and $\eta_{1},\eta_{2}\in \rL^{2}(M)$.

\subsection{Correspondences}
Let $M$ and $N$ be von Neumann algebras. The opposite von Neumann algebra $N\op=\{n\op:n\in N\}$ is identical to $N$ as a Banach space, but with product $n_{1}\op n_{2}\op=(n_{2}n_{1})\op$ and involution $(n\op)^{*}=(n^{*})\op$. An $M$-$N$-correspondence is a $*$-repre\-sen\-tation $\pi_{\cH} \colon M\odot N\op\rightarrow \mathbf{B}(\cH)$ that is normal in each variable, or equivalently a binormal $M$-$N$-bimodule structure on $\cH$.

We will say that an $M$-$N$-correspondence $\cH$ is \emph{contained} in another $\cK$, written abusively $\cH \subset \cK$, if there exists an $M$-$N$-bimodular isometry $V\colon \cH \rightarrow \cK$. We will say that $\cH$ is \emph{weakly contained} in $\cK$, written $\cH\prec\cK$, if we have $\| \pi_{\cH}(T) \| \leq \| \pi_{\cK}(T) \|$ for all $T \in M \odot N\op$. Said another way, we have $\cH \prec \cK$ if and only if the matrix coefficients of $\cH$ can be locally modeled using those of $\cK$: for all $\xi \in \cH$, finite subsets $E\subset M$, $F\subset N$ and  $\varepsilon>0$ there exist $n\in \IN$ and $\{\eta_{1},\ldots \eta_{n}\}\subset \cK$ such that
\begin{align*}
    |\langle x\xi y, \xi \rangle-\sum_{i=1}^{n}\langle x\eta_{i}y,\eta_{i} \rangle|<\varepsilon
\end{align*}
for all $x\in E$ and $y\in F$. If $\cH$ has a vector $\xi$ which is cyclic, i.e.\ such that $M \xi N$ is dense in $\cH$, then it is enough to check the above criterion for $\xi$.

Two $M$-$N$-correspondences $\cH$ and $\cK$ are weakly equivalent, written $\cH\sim\cK$, if $\cH\prec \cK$ and $\cK\prec \cH$. Equivalently, $\cH\sim\cK$ if and only if $\|\pi_{\cH}(T)\|=\|\pi_{\cK}(T)\|$ for all $T \in M \odot N\op$.


\subsection{Intertwining bimodules and finite index subfactors}
We will use the following classical criterion from Popa's intertwining theory 
to reduce our problem to the case of finite index subfactors (\cite{PiPo86}). 
We are interested in $(\rm i) \Rightarrow (\rm iv)$, for which we give 
a direct proof for the reader's convenience. 
\begin{lem}[\cite{Po01,Po03}] \label{intertwining}
Let $M$ be a $\II_1$ factor and $N \subset M$ a subalgebra. The following are equivalent:
\begin{itemize}
\item [$(\rm i)$] $\rL^2(M) \subset \rL^2(\langle M, N \rangle)$ as an $M$-$M$-bimodule.
\item [$(\rm ii)$] There exists a normal conditional expectation $\Phi \colon \langle M, N \rangle \rightarrow M$.
\item [$(\rm iii)$] $M \prec_M N$ in the sense of Popa's intertwining theory.
\item [$(\rm iv)$] There exists a non-zero projection $p \in N' \cap M$ such that $Np \subset pMp$ is an irreducible subfactor with finite index.
\end{itemize}
\end{lem}
\begin{proof}[Proof of $(\rm i) \Rightarrow (\rm iv)$]
Assume $(\rm i)$ holds. 
The normal conditional expectation $\Phi$ in $(\rm ii)$ is obtained by 
compressing elements of $\langle M, N \rangle$ to 
the $M$-$M$-sub\-module $\rL^2(M)\subset \rL^2(\langle M, N \rangle)$. 
Since $Me_NM$ is ultraweakly dense in $\langle M, N \rangle$, 
the element $a:=\Phi(e_N) \in N'\cap M$ is nonzero. 
Let $q\in N'\cap M$ denote the support projection of $a$. 
We claim that $q(N'\cap M)q$ is completely atomic. 
Suppose for the sake of a contradiction that there is 
a nonzero projection $q_0\le q$ in $N'\cap M$ such 
that $q_0(N'\cap M)q_0$ is diffuse. 
Then for any $\ve>0$, there are projections 
$q_1,\ldots,q_n$ in $N'\cap M$ such that $\tau_M(q_i)\le\ve$ and 
$\sum_i q_i=q_0$. 
Put $q_i\op:=Jq_iJ \in M'\cap\langle M,N\rangle$ and 
observe that $\Phi(q_i\op)\in M'\cap M=\IC1$ and $q_ie_N=q_i\op e_N$. 
Hence 
\[
q_0 e_N q_0 = (\sum_i q_iq_i\op) e_N (\sum_j q_jq_j\op) \le \sum_i q_iq_i\op
\]
and so 
\[
(\tau\circ\Phi)(q_0 e_N q_0 ) \le (\tau\circ\Phi)(\sum_i q_iq_i\op) 
= \sum_i\tau_M(q_i)\Phi(q_i\op)\le \ve.
\]
Since $\ve>0$ was arbitrary, one obtains 
$q_0aq_0=\Phi(q_0 e_N q_0 )=0$, a contradiction. 
Thus one can find a minimal projection $p$ in $q(N'\cap M)q$. 
Then, $Np\subset pMp$ is an irreducible subfactor and 
its basic construction is $*$-iso\-morphic to $pp\op \langle M,N\rangle pp\op$ 
via $\rL^2(pMp)\cong pp\op \rL^2(M)$, 
where $p\op:=JpJ \in M'\cap\langle M,N\rangle$. 
Since $\Phi(p\op e_N)=\Phi(p e_N)=pa\neq0$, the element 
$\Phi(p\op)\in M'\cap M=\IC1$ is nonzero and 
$\Phi(p\op)^{-1}p\op\Phi(\,\cdot\,)$ defines a normal conditional expectation 
from $pp\op \langle M,N\rangle pp\op$ onto $pp\op M pp\op\cong pMp$.
This proves that $Np\subset pMp$ has finite index.
\end{proof}

\subsection{Crossed products by locally compact groups}
Let $\sigma \colon G \curvearrowright M$ be a continuous action of a locally compact group $G$ on a von Neumann algebra $M$. Let $M \rtimes_{\sigma, \alg} G$ be the algebraic crossed product. 

Let $\alpha \colon M \rightarrow M \ovt \rL^\infty(G)=\rL^\infty(G,M)$ be the normal $*$-homo\-morphism which sends $x \in M$ to $g \mapsto \sigma_g^{-1}(x)$. Then there is a unique injective $*$-homo\-morphism 
$$\pi \colon M \rtimes_{\sigma, \alg} G \rightarrow M \ovt \B(\rL^2(G))$$
 such that $\pi|_M=\alpha$ and $\pi|_G=1 \otimes \lambda$ for all $g \in G$, where $\lambda \colon G \rightarrow \B(\rL^2(G))$ is the left regular representation of $G$.

By definition, the \emph{crossed product von Neumann algebra} $M \rtimes_\sigma G$ is the unique von Neumann algebra containing $M \rtimes_{\sigma,\alg} G$ as a dense $*$-sub\-algebra and such that $\pi$ extends to a faithful normal $*$-homo\-morphism $\pi \colon M \rtimes_\sigma G \rightarrow M \ovt \B(\rL^2(G))$.

Recall that if $M$ is a type $\III$ factor with faithful, normal semifinite weight $\omega$ and associated modular automorphism group $(\sigma_{t})_{t\in \R}$, then $M\cong (M\rtimes_{\sigma}\R)\rtimes_{\widehat{\sigma}}\R$, and $c(M):=M\rtimes_{\sigma}\R$ is a semifinite von Neumann algebra called the \textit{continuous core} of $M$.
 
Beyond this, in this paper we will only need to know the following fundamental facts:
\begin{itemize}
\item  $\pi(M \rtimes_\sigma G)$ is the fixed point subalgebra of $M \ovt \B(\rL^2(G))$ for the diagonal action $\sigma \otimes \rho \colon G \curvearrowright M \ovt \B(\rL^2(G))$ where $\rho \colon G \curvearrowright \B(\rL^2(G))$ is the action coming from the conjugation by the right regular representation of $G$.
\item There is a (surjective) $*$-iso\-morphism from the basic construction $\langle M \rtimes G, M \rangle$ to $M \ovt \B(\rL^2(G))$ which restricts to $\pi$ on $M \rtimes_\sigma G$.
\item If $G$ is compact, there exists a surjective $*$-homomorphism from $M \rtimes G$ onto the basic construction $\langle M, M^G \rangle$ which restricts to the identity on $M$ \cite[Theorem 4.2]{Pa77}.
\end{itemize}

\subsection{Co-amenable inclusions}

\begin{df}
Let $M \subset \B(H)$ be a von Neumann algebra. We say that a subalgebra $N \subset M$ is \emph{co-amenable} in $M$ if there exists a (not necessarily normal) conditional expectation $\Phi \colon N' \rightarrow M'$.
\end{df}

This property does not depend on the choice of the spatial realization $M \subset \B(H)$. Indeed, any other spatial realization can be obtained from $H$ by amplification and reduction. We then observe that  $\Phi$ extends to a conditional expectation $\Phi \ovt \id \colon N' \ovt \B(K) \rightarrow M' \ovt \B(K)$ for any Hilbert space $K$, and restricts to a conditional expectation $\Phi|_{pN'p} \colon pN'p \rightarrow pM'p$ for any projection $p \in M'$.

We see immediately from the definition that the following properties are satisfied:
\begin{itemize}
\item If $N \subset P$ and $P \subset M$ are co-amenable then $N \subset M$ is co-amenable.
\item If we have $N \subset P \subset  M$ and $N \subset M$ is co-amenable then $P \subset M$ is co-amenable (but $N \subset P$ need not be co-amenable).
\item If $p \in N$ is any projection whose central support in $N$ is $1$, then $pNp \subset pMp$ is co-amenable if and only if $N \subset M$ is co-amenable (because the map $x \mapsto px$ is an isomorphism from $N'$ to $pN'$).
\item $N \subset M$ is co-amenable if and only if there exists a conditional expectation $\Phi \colon \langle M, N \rangle \rightarrow M$. Indeed, by definition, the inclusion $M \subset \langle M, N \rangle$ is anti-isomorphic to $M' \subset N'$ in the standard representation $M \subset \B(\rL^2(M))$.
\end{itemize}

Crossed products by locally compact groups produce natural examples of co-amenable inclusions. For the proofs, we refer to Proposition 2.6 and Proposition 3.4 in \cite{AD95}. We point out that our definition of co-amenability corresponds to what is called \emph{relative injectivity} in \cite{AD95}.
\begin{thm} \label{group co-amenable}
Let $M$ be a von Neumann algebra and $\sigma \colon G \curvearrowright M$ and action of a locally compact group. Then the inclusion $M \rtimes G \subset \langle M \rtimes G, M \rangle$ is co-amenable, i.e.\ there exists a conditional expectation $\Phi \colon M \rtimes G \rightarrow M$.

If $G$ is amenable, then the inclusion $M \subset M \rtimes_\sigma G$ is also co-amenable, i.e.\ there exists a conditional expectation $\Phi \colon \langle M \rtimes G , M \rangle \rightarrow M \rtimes G$.
\end{thm}

Finally, we need the following bimodule characterization of co-amenability. The proof is given in Proposition 2.5 and Proposition 3.6 of \cite{AD95} in the semifinite case and this is enough for our purpose. In the appendix, we include a proof for arbitrary von Neumann algebras.
\begin{thm}
Let $N \subset M$ be an inclusion of von Neumann algebras. Then $N \subset M$ is co-amenable if and only if $\rL^2(M) \prec {_M} \rL^2(\langle M, N \rangle )_M$.
\end{thm}

\subsection{Full Factors}
A factor $M$ is \emph{full} if every central net in $M$ is trivial, i.e.\ if for every bounded net $(x_{i})_{i}$ in $M$ such that $\|\varphi(\cdot x_{i})-\varphi(x_{i}\cdot)\|\rightarrow 0$ for every $\varphi\in M_{*}$, there exists a bounded net $(z_{i})_{i}$ in $\mathbb{C}$ such that $x_{i}-z_{i}1\rightarrow 0$ in the strong operator topology. If $M$ is a $\II_{1}$-factor, then by Corollary 3.8 of \cite{Co74} this is equivalent to the negation of the Murray-von Neumann property $\Gamma$. By \cite[Theorem A]{Ma18b}, a factor $M$ is full if and only if $C^{*}_{\lambda \cdot \rho}(M)$ contains the algebra $\cK(\rL^{2}(M))$ of compact operators. If $M$ has separable predual, $M$ is full if and only if the group of inner automorphisms $\Inn(M)$ is closed in $\Aut(M)$ \cite{Co74}.

\section{Proof of the main theorem}
Our proof of Theorem \ref{co-amenable full} is based on a reduction to the following theorem of Pimsner and Popa which solves the finite index case. The second part of this statement follows implicitly from the proof of \cite[Proposition 1.11]{PiPo86}.
\begin{thm}[{\cite[Proposition 1.11]{PiPo86}}] \label{full finite index}
Let $M$ be a $\II_{1}$ factor and $N \subset M$ an irreducible subfactor with finite index. Then $M$ is full if and only if $N$ is full and in that case, we have $N' \cap M^\omega=\C$ for any ultrafilter $\omega \in \beta \N \setminus \N$.
\end{thm}

In order to reduce the problem to the finite index case, we will need few lemmas. 
The first is the following bimodule interpretation of fullness.

\begin{prop} \label{full bimodule}
Let $M$ be a full factor. Then every $M$-$M$-bimodule that is weakly equivalent to $\rL^2(M)$ must contain $\rL^2(M)$.
\end{prop}
\begin{proof}
Suppose that $M$ is full. Let $\cH$ be an $M$-$M$-bimodule with its binormal representation $\pi \colon M \odot M\op \rightarrow \B(\cH)$. Suppose that $\cH$ is weakly equivalent to $\rL^2(M)$. This means that $\pi$ extends to an isometric $*$-repre\-sen\-tation $\pi \colon C^*_{\lambda \cdot \rho}(M) \rightarrow \B(\cH)$. Take a unit vector $\xi \in \rL^2(M)$. In order to show that $\rL^2(M) \subset \cH$, it is enough to find a unit vector $\eta \in \cH$ such that $\langle a\xi b,\xi \rangle=\langle a \eta b ,\eta \rangle$ for all $a,b \in M$. Let $p$ be the projection on the one-dimensional space spanned by $\xi$. By \cite[Theorem 2.1]{Co75b} and \cite[Theorem A]{Ma18b}, $C^*_{\lambda \cdot \rho}(M)$ contains all the compact operators of $\rL^2(M)$ and in particular, it contains $p$. Since $\pi$ is isometric, $\pi(p)$ is a non-zero projection in $\B(\cH)$ and any unit vector $\eta$ in the range of $\pi(p)$ satisfies the desired equality.
\end{proof}
\begin{rem}
The converse of Proposition \ref{full bimodule} is also true when $M$ has separable predual. Indeed, in that case, if $M$ is not full, then there exists $\theta \in \overline{\Inn}(M) \setminus \Inn(M)$. Then the correspondence $\rL^2(\theta)$ is weakly equivalent to $\rL^2(M)$ but does not contain it.
\end{rem}

In the proof of our main theorem, the following key lemma will be combined with Proposition \ref{full bimodule} in order to obtain the condition in Lemma \ref{intertwining}.$(\rm i)$.

\begin{lem} \label{expectation weakly inner}
Let $(M,\tau)$ be a tracial von Neumann algebra. Suppose that $P \subset M$ is a subalgebra which satisfies $P=(P' \cap M^\omega)' \cap M$ for some ultrafilter $\omega \in \beta \N \setminus \N$. Then $_M\rL^2(\langle M,P\rangle )_M \prec {_M}\rL^2(M)_M$.
\end{lem}
\begin{proof}
Let $\{x_1, \dots, x_n \}$ be a finite subset of $M$. We will use the notations $\underline{M}=M^{\oplus n}$, $\underline{x}=(x_1,\dots,x_n) \in \underline{M}$, $\underline{M}^\omega=(M^\omega)^{\oplus n}=(M^{\oplus n})^\omega$. For every finite set $F \subset P$ and every $\varepsilon > 0$, we define $$U_{F,\varepsilon}=\{ u \in \mathcal{U}(M) \mid \sum_{a \in F} \|ua-au\|_2 < \varepsilon \}$$
and we let
$$ C_{F,\varepsilon}=\conv \{ u \underline{x} u ^* \mid u \in U_{F,\varepsilon} \}, \quad C= \bigcap_{F,\varepsilon} C_{F,\varepsilon}.$$

Then $C$ is a non-empty convex weak$^*$-compact subset of $M$. Let $\underline{y}=(y_1,\dots, y_n) \in C$ be the unique element with minimal $\| \cdot \|_2$-norm. Observe that for every $u \in \mathcal{U}(P' \cap M^\omega)$, we have $\rE(u\underline{y}u^*) \in C$ where $\rE \colon \underline{M}^\omega \rightarrow \underline{M}$ is the canonical conditional expectation. Moreover, we have $\| \rE(u\underline{y}u^*) \|_2 \leq \| u\underline{y}u^*\|_2=\|\underline{y}\|_2$. Therefore, by minimality of $\|\underline{y}\|_2$, we must have $\underline{y}=u\underline{y}u^*=\rE(u\underline{y}u^*)$. This shows that $\underline{y} \in (P' \cap M^\omega)' \cap \underline{M}=\underline{P}$. But since $\underline{y} \in C$, we also have $\rE_P(y_i)=\rE_P(x_i)$ for all $i \leq n$. We conclude that $y_i=\rE_P(x_i)$ for all $i \leq n$. This shows that $\rE_P$ is the pointwise weak$^*$-limit of of convex combinations of inner automorphisms.

Now, $\rL^2(\langle M,P \rangle)$ contains a natural $M$-$M$-cyclic vector $\xi$ which satisfies $\langle x \xi y,\xi \rangle=\tau(x\rE_P(y))$ for all $x,y \in M$. But we have proved that $\tau(x\rE_P(y))$ can be approximated by convex combinations of $\tau(xuyu^*)=\langle x (u\hat{1})y, (u\hat{1}) \rangle$ where $\hat{1} \in \rL^2(M)$ and $u \in \mathcal{U}(M)$. Thus $_M \rL^2(\langle M ,P \rangle)_M \prec {_M} \rL^2(M)_M$.
\end{proof}

Finally, we will need the following proposition for the reduction to the tracial case.
\begin{prop} \label{core co-amenable}
Let $N \subset M$ be an inclusion of von Neumann algebras with a faithful normal conditional expectation $\rE_N \colon M \rightarrow N$ and let $c(N) \subset c(M)$ be the associated inclusion of continuous cores. If $N \subset M$ is co-amenable then $c(N) \subset c(M)$ is also co-amenable.
\end{prop}
\begin{proof}
Since $N \subset M$ is co-amenable and $M \subset c(M)$ is co-amenable by Theorem \ref{group co-amenable}, we know that $N \subset c(M)$ is co-amenable. Since $N \subset c(N) \subset c(M)$, it follows that $c(N) \subset c(M)$ is co-amenable.
\end{proof}

\begin{proof}[Proof of Theorem \ref{co-amenable full}]
We first deal with the case where $M$ is a $\II_1$ factor. Fix a free ultrafilter $\omega \in \beta \N \setminus \N$. Let $P=(N' \cap M^\omega)' \cap M$. Then we have $P=(P' \cap M^\omega)' \cap M$. By Lemma \ref{expectation weakly inner}, we have $_M\rL^2(\langle M,P \rangle )_M \prec {_M}\rL^2(M)_M$. Since $N \subset M$ is co-amenable and $N \subset P$, we have that $P \subset M$ is still co-amenable. Therefore  $_M\rL^2(M)_M \prec {_M}\rL^2(\langle M,P \rangle )_M$. This shows that the $M$-$M$-bimodule $_M\rL^2(\langle M,P \rangle )_M$ is weakly equivalent to $_M\rL^2(M)_M$. Therefore, by Proposition \ref{full bimodule}, we have $_M\rL^2(M)_M \subset {_M}\rL^2(\langle M,P \rangle )_M$. Therefore, by Lemma \ref{intertwining}, there exists a non-zero $p \in P' \cap M$ such that $Pp \subset pMp$ is an irreducible subfactor with finite index. Since $pMp$ is full, we obtain $p(P' \cap M^\omega) p = \C p$ by Theorem \ref{full finite index}. In particular, we have $p \in N' \cap M$ and $p(N' \cap M^\omega)p=\C p$.

Now suppose that $M$ is of arbitrary type. Let $K$ be a full type $\III_1$ factor with separable predual and with trivial $\tau$-invariant (for example take $K$ the free Araki-Woods factor associated to the regular orthogonal representation of $\R$, see \cite{Sh98}). Let $M_1=M \ovt K$ and $N_1=N \ovt K$. Then by \cite{HMV16}, $M_1$ is a full type $\III_1$ factor with trivial $\tau$-invariant. Therefore by \cite{Ma16}, we know that the continuous core $c(M_1)$ is a full $\II_\infty$ factor. Now, pick a faithful normal state $\varphi$ on $N$ and extend it to $M$ by using a faithful normal conditional expectation $\rE_N\colon M \rightarrow N$. Define a faithful normal state on $M_1$ by $\varphi_1 =\varphi \otimes \psi$ where $\psi$ is any faithful normal state on $K$. Then we have a trace preserving inclusion of the continuous cores $c(N_1)=N_1 \rtimes_{\sigma^{\varphi_1}} \R \subset  c(M_1)=M_1 \rtimes_{\sigma^{\varphi_1}} \R$. The inclusion $c(N_1) \subset c(M_1)$ is co-amenable by Proposition \ref{core co-amenable}. Pick a non-zero finite trace projection $q \in L(\R)=\C \rtimes_\varphi \R \subset c(N_1)$. Observe that the central support of $q$ in $c(N_1)$ is $1$ because $N_1$ is of type $\III_1$. Therefore $qc(N_1)q$ is still a co-amenable subalgebra in the full $\II_1$ factor $qc(M_1)q$. Thus, by the first part of the proof, there exists a non-zero projection $p \in c(N_1)' \cap c(M_1)$ such that $pq(c(N_1)' \cap c(M_1)^\omega)pq=\C pq$. Since $q$ has central support $1$ in $c(N_1)$, this implies $p(c(N_1)' \cap c(M_1)^\omega)p=\C p$. Now, observe that $c(N_1)' \cap c(M_1)=N' \cap M_\varphi$ because $K$ is of type $\III_1$. Moreover, $N' \cap M^\omega_{\varphi^\omega} \subset c(N_1)' \cap c(M_1)^\omega$. Therefore, we have $p \in N' \cap M_\varphi$ and $p(N' \cap M^\omega_{\varphi^\omega})p=\C p$. Since $p(N' \cap M^\omega_{\varphi^\omega})p$ is the centralizer of the state $\varphi^\omega$ restricted to  $p(N' \cap M^\omega)p$, we know by \cite[Lemma 5.3]{AH12} that $p(N' \cap M^\omega)p$ is either trivial or a type $\III_1$ factor. But the latter case cannot happen because if $p(N' \cap M^\omega)p$ is of type $\III_1$, then for any $\varepsilon > 0$, we can find a Haar unitary $u \in p(N' \cap M^\omega)p$ such that $\| u\varphi^\omega-\varphi^\omega u\| \leq \varepsilon$ and by a diagonal extraction, this would imply that $p(N' \cap M^\omega_{\varphi^\omega})p \neq \C$. Thus we conclude that $p(N' \cap M^\omega)p= \C p$.
\end{proof}
\begin{rem} \label{rem spectral gap}
Note that the \emph{w-spectral gap} property for an inclusion of factors is weaker than the \emph{spectral gap} property in general (we refer the reader to the introduction of \cite{IV15} for the definitions and a discussion of this subtle difference). Thus, one might wonder wether in Corollary \ref{cor w-gap}, one can improve the w-spectral gap to a true spectral gap. The following example shows that the answer is negative in general. 

Let $\F_2 \curvearrowright (X,\mu)$ be a free ergodic pmp action which is \emph{weakly compact} (see \cite[Section 3]{OP07}) and strongly ergodic but does not have spectral gap. The weak compactness implies that the inclusion $L(\F_2) \subset \rL^\infty(X) \rtimes \F_2$ is co-amenable by \cite[Proposition 3.2]{OP07}. Since the action $\F_2 \curvearrowright (X,\mu)$ is strongly ergodic and $\F_2$ is non-inner amenable, the crossed product $\rL^\infty(X) \rtimes \F_2$ is full by \cite{Ch81}. However, since $\F_2 \curvearrowright (X,\mu)$ does not have spectral gap, the subfactor $L(\F_2)$ does not have spectral gap in $\rL^\infty(X) \rtimes \F_2$ (see the introduction of \cite{IV15}).

In order to construct such an action $\F_2 \curvearrowright (X,\mu)$, one can use the result \cite[Theorem A.3.2]{HK05} which was pointed to us by Adrian Ioana. We start from a compact free ergodic pmp action $\F_2=\langle a,b \rangle \curvearrowright (Y,\nu)$ which has spectral gap and such that $a$ acts ergodically on $(Y,\nu)$. Then we use \cite[Theorem A.3.2]{HK05} to obtain a new action $\F_2 \curvearrowright (X,\mu)$ which is orbit equivalent to $\F_2 \curvearrowright (Y,\nu)$ but does not have spectral gap. This new action $\F_2 \curvearrowright (X,\mu)$ is strongly ergodic and also weakly compact by \cite[Section 6]{Io08}, as we wanted.
\end{rem}

\section{Actions of compact groups on full factors}
Recall that an action $\sigma \colon G \curvearrowright M$ of a locally compact group $G$ on a factor $M$ is called \emph{outer} if $\sigma_g \notin \Inn(M)$ for all $g \in G \setminus \{1\}$. It is called \emph{strictly outer} if $M' \cap (M \rtimes_\sigma G)=\C$. Finally, it is called \emph{minimal} if it is faithful and $(M^G)' \cap M=\C$. A minimal action is strictly outer and a strictly outer action is outer, but the reverse implications are not true in general. When $G$ is discrete, $\sigma$ is outer if and only if it is strictly outer. When $G$ is compact, $\sigma$ is minimal if and only if it is strictly outer. We refer the reader to \cite{Va05} for further details regarding these notions.

In order to prove Theorem \ref{minimalGaction}, we will first need to prove the following dichotomy:
\begin{thm} \label{minimal or diffuse}
Let $\sigma \colon G \curvearrowright M$ be an action of a compact group $G$ on a factor $M$. Suppose that $\sigma$ is outer. Then $(M^G)' \cap M$ is either trivial or diffuse.
\end{thm}

The following lemma is an easy consequence of the Peter-Weyl theorem.
\begin{lem}
Let $M$ be a von Neumann algebra. The following are equivalent:
\begin{itemize}
\item [$(\rm i)$] $M$ is a multi-matrix algebra, i.e.\ $M \cong \bigoplus_{i \in I} M_{n_i}(\C)$ for some (possibly infinite) family of integers $n_i \geq 1, \: i \in I$.
\item [$(\rm ii)$] The unitary group of $M$ is compact.
\item [$(\rm iii)$] $M$ is generated by a compact group of unitaries.
\end{itemize}
\end{lem}

\begin{lem} \label{compact multi-matrix}
Let $\sigma \colon G \curvearrowright M$ be an action of a compact group $G$ on a factor $M$. Then the following are equivalent.
\begin{itemize}
\item [$(\rm i)$] $M' \cap (M \rtimes G)$ is a multi-matrix algebra.
\item [$(\rm i)'$] $(M^G)' \cap M$ is a multi-matrix algebra.
\item [$(\rm ii)$] $M' \cap (M \rtimes G)$ is completely atomic.
\item [$(\rm ii)'$] $(M^G)' \cap M$ is completely atomic.
\item [$(\rm iii)$] $M' \cap (M \rtimes G)$ is not diffuse.
\item [$(\rm iii)'$] $(M^G)' \cap M$ is not diffuse.
\end{itemize}
\end{lem}
\begin{proof}
We will fix the notation $N=M \ovt \B(\rL^2(G))$ and consider the action $\sigma \otimes \rho  \colon G \curvearrowright N$ where $\rho \colon G \curvearrowright \B(\rL^2(G))$ is the action coming from the right regular representation as in the preliminary section. Recall that we have an isomorphism $\pi \colon M \rtimes_\sigma G \rightarrow N^G$ which extends to an isomorphism from $\langle M \rtimes_\sigma G, M \rangle$ to $N$. Since $G$ is compact, we also observe that if $e_G \in \B(\rL^2(G))$ is the rank-one projection on the scalar functions, then $e_G$ is invariant under $\rho$ and $\sigma \otimes \rho$ restricts to $\sigma$ on $e_GNe_G=M$.

The implications $(\rm i) \Rightarrow (\rm ii) \Rightarrow (\rm iii)$ and $(\rm i)' \Rightarrow (\rm ii)' \Rightarrow (\rm iii)'$ are obvious. The implications $(\rm i) \Rightarrow (\rm i)'$, $(\rm ii) \Rightarrow  (\rm ii)'$ and $(\rm iii)' \Rightarrow (\rm iii)$ are also easy. Indeed, this follows from \cite[Theorem 4.2]{Pa77} which says that there is a surjective $*$-homomorphism from $M \rtimes G$ onto $\langle M, M^G \rangle$ which is the identity on $M$. Hence $M' \cap (M \rtimes G)$ surjects onto $M' \cap \langle M, M^G \rangle \cong ((M^G)' \cap M)^{\op}$. We conclude that if $M' \cap (M \rtimes G)$ is a multi-matrix algebra (resp.\ completely atomic, resp.\ diffuse) then $(M^G)' \cap M$ is also a multi-matrix algebra (resp.\ completely atomic, resp.\ diffuse).


$(\rm i)' \Rightarrow  (\rm i)$. Observe that $(M^G)' \cap (M \rtimes G)=((M^G)' \cap M) \rtimes G$ (this follows from the fact that $\pi(M^G)=M^G \otimes \C \subset M \ovt \B(\rL^2(G))$). Since $(M^G)' \cap M$ is a multi-matrix algebra then its unitary group $\mathcal{U}((M^G)' \cap M)$ is compact. Therefore, $(M^G)' \cap M) \rtimes G$ is generated by the compact group $\mathcal{U}((M^G)' \cap M) \rtimes G$, hence $(M^G)' \cap (M \rtimes G)$ is a multi-matrix algebra. In particular $M' \cap (M \rtimes G)$ is a multi-matrix algebra.

$(\rm ii)' \Rightarrow (\rm i)'$. This can be deduced from \cite{HLS81} but we provide an easy proof in our specific situation. Take $p$ a minimal projection in the center of $M^G$. Then the action of $G$ on the algebra $D=((M^G)' \cap M))p$ is ergodic, i.e.\ $D^G=\C$. Since $D$ is completely atomic, there exists a unique faithful semifinite trace $\tau$ on $D$ such that $\tau(e)=1$ for every minimal projection $e \in D$. Since $\tau$ is unique, it must be $G$-invariant. But the canonical $G$-invariant conditional expectation $E \colon M \rightarrow M^G$ already restricts to a $G$-invariant state $\psi$ on $D$. Since $D^G=\C$, this forces $\tau$ to be proportional to $\psi$. In particular, $\tau$ is finite and $D$ must be finite dimensional. This holds for every minimal projection $p \in \mathcal{Z}(M^G)$ hence $(M^G)' \cap M$ is a direct sum of finite dimensional algebras.

$(\rm iii)' \Rightarrow (\rm ii)$. Let $z$ be the maximal central projection in $(M^G)' \cap M$ such that $((M^G)' \cap M)z$ is completely atomic. Then $z$ must be $G$-invariant, hence $z \in \mathcal{Z}(M^G)$. Consider the truncated action $\sigma' \colon G \curvearrowright Q=zMz$. Then $(Q^G)' \cap Q$ is completely atomic. Hence, by $(\rm ii)' \Rightarrow (\rm i)' \Rightarrow (\rm i) \Rightarrow (\rm ii)$, we have that $Q' \cap (Q \rtimes G)$ is completely atomic. But we have $Q' \cap (Q \rtimes G)=(M' \cap (M \rtimes G))z$ and $(M' \cap (M \rtimes G))z$ is isomorphic to $M ' \cap (M \rtimes G)$ because $z \in M$ and $M$ is a factor. Therefore $M' \cap (M \rtimes G)$ is completely atomic.

$(\rm iii) \Rightarrow (\rm ii)$. If $M' \cap (M \rtimes G)$ is not diffuse, then $(N^G)' \cap N$, which is anti-isomorphic to it, is also not diffuse. Then the implications $(\rm iii)' \Rightarrow (\rm ii) \Rightarrow (\rm ii)'$ show that $(N^G)' \cap N$ is completely atomic, hence $M' \cap (M \rtimes G)$ is completely atomic.

The reader may now check that all the implications we proved form a connected graph, so that all the properties are equivalent.
\end{proof}

\begin{proof}[Proof of Theorem \ref{minimal or diffuse}]
 
Suppose that $(M^G)' \cap M$ is not diffuse. By Lemma \ref{compact multi-matrix}, we know that $(M^G)' \cap M$ is a multi-matrix algebra. We want to show that $\sigma$ is minimal, i.e.\ that $(M^G)' \cap M=\C$.

We first deal with the case where $(M^G)' \cap M$ is a factor. In this case, since $Q=(M^G)' \cap M$ is a finite dimensional factor, we have $M=M_0 \otimes Q$ where $M_0 =Q' \cap M$. Observe that the action $\sigma_0 \colon G \curvearrowright M_0$ is minimal. Indeed, by construction we have $M^G=(M_0)^G$ and $(M^G)' \cap M_0=\C$, and $\sigma_0$ is outer (hence faithful) because $\sigma$ is outer. Now, suppose that $Q \neq \C$ and consider the representation of $G$ on $Q \ominus \C$ induced by the action $G \curvearrowright Q$. Choose $V \subset Q \ominus \C$ an irreducible subrepresentation. Since $\sigma_0$ is minimal, we can find a copy of the contragredient representation $\overline{V} \subset M_0$. Then $\overline{V}  \otimes V \subset M_0 \otimes (Q \ominus \C)=M \ominus M_0$ contains a copy of the trivial representation of $G$. But this contradicts the fact that $M^G \subset M_0$. We conclude that $Q=\C$, hence $\sigma$ is minimal.


Now, we deal with the general case where $(M^G)' \cap M$ is not necessarily a factor. Among all open subgroups $K \subset G$ and all non-zero projections $e \in M^K$, we choose a pair $(K,e)$ such that the dimension of $e((M^K)' \cap M)e$ is minimal. Take $f$ a minimal projection in the center of $e((M^K)' \cap M)e$. Then there exists an open subgroup $H \subset K$ which stabilizes $f$. Since $f((M^H)' \cap M)f \subset e((M^K)' \cap M)e$, by minimality of $(K,e)$, we must have $f((M^H)' \cap M)f =e((M^K)' \cap M)e$ hence $f=e$. This shows that $e((M^K)' \cap M)e=((eMe)^K)' \cap eMe$ is a factor. But the action of $K$ on $eMe$ is outer (because $K \curvearrowright M$ is outer and $M$ is a factor). Since we already treated the factorial case, we get that $K \curvearrowright eMe$ is minimal. Therefore $K \curvearrowright eMe$ is strictly outer, i.e.\ $(M' \cap (M \rtimes K))e=(eMe)' \cap (eMe \rtimes K)=\C e$. Since $M$ is a factor, this implies that $M' \cap (M \rtimes K)=\C$, i.e.\ $K \curvearrowright M$ is strictly outer. Finally, using the fact that $K$ is open in $G$, we will conclude that $G \curvearrowright M$ is also strictly outer, hence minimal. Indeed, since $K$ is open, any $x \in M \rtimes G$ has a ``Fourier decomposition" $x=\sum_{g \in G/K} u_g x_g$ where $x_g \in M \rtimes K$ (we implicitly choose a set of representatives of each equivalence class in $G/K$). We refer to \cite{BB16} for an explanation of this ``Fourier decomposition". If $x \in M' \cap (M \rtimes G)$, we must have $x_g a =\sigma^{-1}_g(a)x_g$ for all $a \in M$. By \cite[Corollary 3.11]{BB16}, and since $\sigma$ is outer, this is only possible if $g \in K$. Therefore $M' \cap (M \rtimes G) = M' \cap (M \rtimes K)=\C$ as we wanted.
\end{proof}

\begin{rem} Let $\sigma \colon G \curvearrowright M$ be an action of a compact group $G$ on a factor $M$. Let $H=\{g \in G \mid \sigma_g \text{ is inner} \}$. Then one can deduce from Theorem \ref{minimal or diffuse} that either $(M^G)' \cap M$ is diffuse or $$(M^G)' \cap M=(M^H)' \cap M=\{ u \in M \mid \exists g \in G, \: \Ad(u)=\sigma_g \}''.$$
\end{rem}

\begin{proof}[Proof of Theorem \ref{minimalGaction}]
Suppose first that $M^G \subset M$ is co-amenable. Since $M$ is full, then by Theorem \ref{co-amenable full}, $(M^G)' \cap M$ is not diffuse and by combining with Theorem \ref{minimal or diffuse} it must be trivial. We conclude that $(M^G)' \cap M^\omega=\C$ by Theorem \ref{co-amenable full}. Next, since we have $\pi(M \rtimes G)$ is with expectation in $M \ovt \B(\rL^2(G))$, we have $\pi(M \rtimes G)^\omega \subset M^\omega \ovt \B(\rL^2(G))$. Since $\pi(M^G)=M^G \otimes \C$, we get $\pi (M ' \cap (M \rtimes G)^\omega) \subset ((M^G)' \cap M^\omega) \ovt \B(\rL^2(G))=\C \otimes \B(\rL^2(G))$. We conclude that $M' \cap (M \rtimes G)^\omega=M' \cap (M \rtimes G)=\C$.

Now, in the general case, consider the action $\sigma \otimes \rho \colon G \curvearrowright N=M \ovt \B(\rL^2(G))$. Then the inclusion $N^G \subset N$ which is isomorphic to $M \rtimes G \subset \langle M \rtimes G, M\rangle$  is co-amenable by Theorem \ref{group co-amenable}. So we can apply the first case to deduce that $\sigma \otimes \rho$ is minimal, hence $\sigma$ is also minimal. In particular, $M^G \subset M$ is co-amenable and we conclude by the first case.
\end{proof}

\appendix

\section{Bimodule characterization of co-amenability}

Let $M$ be a von Neumann algebra. 
We will denote by $\rL^2(M)$ the standard Hilbert space for $M$,  
by $J_M$ the modular conjugation, by $\rL^2(M)^+$ the natural positive cone, 
and by $\pi_M \colon M  \odot M\op \rightarrow \B(\rL^2(M))$ the natural binormal $*$-repre\-sen\-tation. We define the norm 
$$ M \odot M\op \ni T \mapsto \| T \|_{\pi_M}=\| \pi_M(T) \|_{\B(\rL^2(M))}.$$
For $x \in M$, we will use the notation $\overline{x}=(x^*)\op=(x\op)^* \in M\op$. 

Haagerup has made a deep study of the norm 
\[
M^{\oplus n} \ni (x_1, \ldots, x_n) \mapsto  \|\sum_i x_i \otimes \overline{x_i} \|_{\pi_M}
\]
and obtained many remarkable results in \cite{Ha93}, 
but we need only the following (Theorem 2.1 in \cite{Ha93}). 
This was proved earlier in the semifinite case by Pisier (Theorem 2.1 in \cite{Pi95}).
Alternative proofs are found in \cite{JP10} and \cite{Xu05}. 

\begin{thm}[Pisier--Haagerup]\label{thm:haagerup}
For any von Neumann algebra $M$ and any $x_1,\ldots,x_n\in M$,  
\[  
\|\sum_i x_i \otimes \overline{x_i} \|_{\pi_M}^{1/2}=\|(x_1,\ldots,x_n)\|_{1/2},
\]
where the latter norm is the complex interpolation of equivalent norms 
$\|\,\cdot\,\|_{0}$ and $\|\,\cdot\,\|_{1}$ 
on $M^{\oplus n}$ that are given by  
\[
\|(a_1,\ldots,a_n)\|_{0}:=\|\sum_i a_ia_i^*\|^{1/2}\mbox{ and }
\|(a_1,\ldots,a_n)\|_{1}:=\|\sum_i a_i^*a_i\|^{1/2}.
\]
In particular, for any (not necessarily normal) 
unital completely positive map $\theta$ from 
$M$ into a von Neumann algebra $N$ and any $x_1,\ldots,x_n\in M$, 
\[
 \|\sum_i \theta(x_i) \otimes \overline{\theta(x_i)} \|_{\pi_N}
\le  \|\sum_i x_i \otimes \overline{x_i} \|_{\pi_M}.
\]
\end{thm}

Note that the second half of Theorem~\ref{thm:haagerup} is 
a trivial consequence of the interpolation theorem, because 
$\theta\otimes\id_n\colon (M^{\oplus n},\|\,\cdot\,\|_{t})
 \to(N^{\oplus n},\|\,\cdot\,\|_{t})$ 
is contractive at $t=0,1$. 
The purpose of this appendix is to give a proof of 
the following corollary to Theorem~\ref{thm:haagerup}.

\begin{cor}[cf.\ Corollary 3.8 in \cite{Ha93}]\label{cor:haagerup}
For von Neumann algebras $N\subset M$, the following are equivalent. 
\begin{enumerate}[$(\rm i)$]
\item\label{corcond1}
There is a conditional expectation of $M$ onto $N$.
\item\label{corcond2}
For every $x_1,\ldots,x_n\in N$,
\[
\|\sum_i x_i \otimes \overline{x_i} \|_{\pi_N}=\|\sum_i x_i \otimes \overline{x_i} \|_{\pi_M}. 
\]
\item\label{corcond3}
${}_N\rL^2(N)_N\prec{}_N\rL^2(M)_N$. Namely, for every $x_i,y_i\in N$, 
\[
\|\sum_i x_i \otimes \overline{y_i} \|_{\pi_N} \le \|\sum_i x_i \otimes \overline{y_i} \|_{\pi_M}. 
\]
\end{enumerate}
\end{cor}

The following reformulation of Corollary~\ref{cor:haagerup} 
closes the circle of implications of Propositions 2.5 in \cite{AD95}.

\begin{cor}[cf.\ Proposition 3.6 in \cite{AD95}]\label{cor:anantharaman}
Let ${}_M\cH_N$ be an $M$-$N$ correspondence and 
put $N_1=\B_{N\op}(\cH)$, the commutant of the right $N$-action.
Then the following are equivalent. 
\begin{enumerate}[$(\rm i)$]
\item ${}_M\cH_N$ is left-amenable, i.e.\ ${}_M \rL^2(M)_M\prec{}_M \rL^2(N_1)_M$.
\item There is a conditional expectation of $N_1$ onto $M$.
\end{enumerate}
\end{cor}

We recall the notion of the \emph{selfpolar forms} (\cite{Co74b,Wo74}), 
which plays a central role in \cite{Ha93} as well 
as in the proof of Corollary~\ref{cor:haagerup}.
Associated with any normal state $\vp$ on $M$, there 
are a canonical unit vector $\xi_\vp$ in $\rL^2(M)^+$ 
and the selfpolar form $s_\vp\colon M\times M\to\IC$ given by 
\[
s_\vp(x,y)=\ip{\pi_M(x \otimes \overline{y})\xi_\vp,\xi_\vp}_{\rL^2(M)}.
\]

\begin{prop}[Theorems 1.1 and 2.1 in \cite{Wo74}]\label{prop:woronowicz}
The selfpolar form $s_\vp$ satisfies the following properties. 
\begin{enumerate}[$(\rm i)$]
\item
$s_\vp$ is a positive semi-definite sesqui-linear form on $M\times M$,
\item
$s_\vp\geq0$ on $M_+\times M_+$,
\item
$s_\vp(x,1)=\vp(x)$ for $x\in M$, and
\item 
$s_{\vp}$ is selfpolar.
\end{enumerate}
Moreover, if $s$ is any sesqui-linear form on $M$ which satisfies 
(i) - (iii) above (with $s_\vp$ exchanged by $s$), 
then $s(x,x)\le s_\vp(x,x)$ for every $x\in M$.
\end{prop}

We do not introduce the selfpolar property (iv), 
since we will not use it (in an explicit way). 
We use the following well-known variant of the Hahn--Banach theorem. 

\begin{lem}\label{lem:hb}
Let $C$ be a convex cone of a real vector space $V$, 
$p\colon V\to\R$ a sublinear function, and $q\colon C\to\R$ a superlinear function.
This means that $p(x+y)\le p(x)+p(y)$ and 
$p(\lambda x)=\lambda p(x)$ for every $x,y\in V$ and $\lambda\in\R_{\geq0}$; 
and that $-q$ is sublinear. 
If $q\le p$ on $C$, then there is a linear functional $\psi$ on $V$ 
such that $\psi\le p$ on $V$ and $q\le\psi$ on $C$. 
\end{lem}
\begin{proof}
Observe that $r(x):=\inf\{ p(x+y)-q(y) : y\in C\}$ is a sublinear function on $V$ 
such that $-p(-x) \le r(x) \le p(x)$ for $x\in V$ 
(in particular $r$ takes values in $\R$).
By the Hahn--Banach theorem, there is a linear functional $\psi$ on  $V$ 
such that $\psi\le r$. 
One has $-\psi(x)\le r(-x)\le -q(x)$ for $x\in C$. 
\end{proof}

\begin{proof}[Proof of Corollary~\ref{cor:haagerup}]
The equivalence $({\rm\ref{corcond1}})\Leftrightarrow({\rm\ref{corcond2}})$ 
is proved in Corollary 3.8 in \cite{Ha93}.
That $({\rm\ref{corcond1}})\Rightarrow({\rm\ref{corcond2}})$ follows from 
Theorem~\ref{thm:haagerup}. 
That $({\rm\ref{corcond3}})\Rightarrow({\rm\ref{corcond1}})$ is proved in 
Proposition 2.5 in \cite{AD95}.
Thus it is left to show $({\rm\ref{corcond2}})\Rightarrow({\rm\ref{corcond3}})$. 
We closely follow the proof of Corollary 3.8 (Theorem 3.7) in \cite{Ha93}.
Take any normal state $\vp$ on $N$ and consider the convex cone 
$C:=\{\sum_i x_i\otimes\overline{x_i} : x_i \in N \}\subset M\odot M\op$ 
and the seminorm $p(\,\cdot\,)=\|\pi_M(\,\cdot\,)\|_{\B((\rL^2(M))}$ on $M\odot M\op$. 
Since 
\[
\sum_i s_{\vp}(x_i,x_i) 
 = \ip{\pi_N(\sum_i x_i\otimes\overline{x_i})\xi_\vp,\xi_\vp}
 \le p(\sum_i x_i\otimes\overline{x_i})
\]
by condition (\ref{corcond2}), 
Lemma~\ref{lem:hb} yields $\psi\in \B(\rL^2(M))^*$ of norm $1$ such that 
$s_{\vp}(x,x)\le \Re\psi(\pi_M(x \otimes \overline{x}))$ for $x\in N$.
Observe that $\|\psi\|\le1$ and $1\le\Re\psi(1)$ imply that $\psi$ is a state. 
The state $\psi^*$ on $\B(\rL^2(M))$ defined by $\psi^*(T)=\psi(J_M T^* J_M)$
satisfies $\psi^*(\pi_M(x \otimes \overline{y}))=\overline{\psi(\pi_M(y \otimes \overline{x}))}$ 
for every $x,y\in M$. 
Thus by replacing $\psi$ with the state $\frac{1}{2}(\psi+\psi^*)$, we may 
assume that $s\colon M\times M\ni (x,y)\mapsto\psi(\pi_M(x \otimes \overline{y}))$ is a sesqui-linear 
form such that $s_{\vp}(x,x) \le s(x,x)$ for every $x\in N$.
Since $s_{\vp}(1+\lambda x,1+\lambda x) \le s(1+\lambda x,1+\lambda x)$ 
for all $\lambda\in\R$ and a fixed $x\in N$, 
one sees $\vp(x)=s(x,1)$ for $x\in N$. 
It follows from Proposition~\ref{prop:woronowicz} 
that $s_{\vp}(x,x)=s(x,x)$ for every $x\in N$.
Hence by polarization identity, $s_{\vp}=s$ and so for every $T=\sum_i x_i \otimes \overline{y_i} \in N \odot N\op$ we have
\[
\langle \pi_N(T) \xi_\varphi,\xi_\varphi \rangle= \sum_i s_\vp(x_i,y_i)=\sum_i s(x_i,y_i) =\psi(\pi_M(T)).
\]
Since $\{ \xi_\vp\}_\vp$ is a $\pi_N$-cyclic family, 
one obtains $$\| \pi_N(T) \|_{\B(\rL^2(N))} \leq \| \pi_M(T) \|_{\B(\rL^2(M))}$$ 
which is precisely condition (\ref{corcond3}). 
\end{proof}

\bibliographystyle{plain}

\end{document}